\documentclass[a4paper,reqno]{amsart}

\usepackage{microtype}
\usepackage{amsmath}
\usepackage{amssymb}
\usepackage{url}
\usepackage{fullpage}
\usepackage{cleveref}
\usepackage{color}
\usepackage{enumitem}
\usepackage{txfonts}
\usepackage[foot]{amsaddr}

\newtheorem{theorem}{Theorem}[section]
\newtheorem{lemma}[theorem]{Lemma}
\newtheorem{proposition}[theorem]{Proposition}
\newtheorem{corollary}[theorem]{Corollary}
\newtheorem{example}[theorem]{Example}
\newtheorem{definition}[theorem]{Definition}
\newtheorem{remark}[theorem]{Remark}

\Crefname{theorem}{Theorem}{Theorems}
\Crefname{corollary}{Corollary}{Corollarys}
\Crefname{lemma}{Lemma}{Lemmas}
\Crefname{proposition}{Proposition}{Propositions}
\Crefname{definition}{Definition}{Definitions}
\Crefname{example}{Example}{Examples}
\Crefname{remark}{Remark}{Remarks}
\Crefname{figure}{Figure}{Figures}

\begin{document}

\allowdisplaybreaks

\title[Embedding the symbolic dynamics of Lorenz maps]{Embedding the symbolic dynamics of Lorenz maps}

\author{T.~Samuel}
\address{Fachbereich 3 -- Mathematik, Universit\"at Bremen, 28359 Bremen, Germany.}
\author{N.~Snigireva}
\address{School of Mathematical Sciences, University College Dublin, Belfield, Dublin 4, Ireland.}
\author{Andrew Vince}
\address{Department of Mathematics, University of Florida, Gaineville, Florida, USA.}

\begin{abstract}
Necessary and sufficient conditions for the symbolic dynamics of a Lorenz map to be fully embedded in the symbolic dynamics of a piecewise continuous interval map are given.  As an application of this result, we describe a new algorithm for calculating the topological entropy of a Lorenz map.
\end{abstract}

\maketitle

\vspace{-1.5em}

\section{Introduction}

Lorenz maps and their topological entropy have been and still are investigated intensively, see for instance \cite{BHV,G,GH,GC,H,HR,HS} and references therein. The simplest example of a Lorenz map is a $\beta$-transformation. The topological entropy of such transformation is well known \cite{Par}. However, for a general Lorenz map the question of determining the topological entropy is much more complicated.  Glendinning \cite{G} showed that every Lorenz map is semi-conjugate to a \mbox{$\beta$-transformation} and thus some features of a Lorenz map can be understood via \mbox{$\beta$-transformations}. In this paper, we  investigate the relation between the symbolic dynamics of a given Lorenz map and that of a $\beta$-transformation.  In particular, this will allow us to obtain upper and lower bounds on the entropy of a general Lorenz map.  Let us now outline our main results.
\begin{description}
\item[(i) \textit{Embedding dynamics}:] Our main results, \Cref{Main_THM,THM:END}, give necessary and sufficient conditions for when the address space (\Cref{Defn:Itinerary_map3}) of an arbitrary Lorenz system is a forward shift sub-invariant subset (\Cref{DEFN:SHIFT_INV}) of the address space of a uniform Lorenz system.  (See \Cref{Defn:Lorenz_map} for the definition of a Lorenz system.)  These results complement those of  
\cite{BHV,GH,HS}.
\item[(ii) \textit{An algorithm}:] Based on (i), we provide, in \Cref{Sec:Algorithm}, an algorithm for calculating the topological entropy of a Lorenz system.  This algorithm does not require previously used techniques of finding zeros of a power series \cite{AM,BHV,GH} nor the calculation of the zero of a pressure functional \cite{FMT}.
\end{description} 

\subsection{Motivation and previous related results}

A main motivation for the study of Lorenz maps is that they arise naturally in the investigation of a geometric model of Lorenz differential equations which have strange attractors, see \cite{Eckhardt,Lorenz:1963,Vis,Wil} and references therein.   A second motivation is that a $\beta$-transformation (being the simplest example of a Lorenz map) plays an important role in ergodic theory, see \cite{DK,G,H,Par} and references therein. A third motivation comes from the study of fractal transformation, see \cite{BHI}.

Results from kneading theory are used in the study of Lorenz maps. In 1990, Hubbard and Sparrow \cite{HS} showed that the upper and lower itineraries of the critical point fully determine the address space of a Lorenz map. Moreover, Glendinning and Hall \cite{GH} showed that the topological entropy of such a map is related to the largest positive zero of a certain power series.  Further results on the kneading sequences of Lorenz maps can be found, for instance, in the works of Hofbauer and Raith \cite{H,HR}, Alsed\'{a} and Ma\~{n}os \cite{AM}, Misiurewicz \cite{M} and Glendinning, Hall and Sparrow \cite{G,GH,GC}.

\subsection{Main results}

To formally state our main results we require the following notation and definitions.

\begin{definition}\label{Defn:Lorenz_map}
An \textit{upper} (or \textit{lower}) \textit{Lorenz map} with \textit{critical point} $q \in (0, 1)$ is a piecewise continuous map $T^{+}$ (respectively $T^{-}$) $: [0, 1] \circlearrowleft$ of the form
\begin{align*}
T^{+}(x) \coloneqq
\begin{cases}
f_{0}(x) & \text{if} \; 0 \leq x < q,\\
f_{1}(x) & \text{if} \; q \leq x \leq 1,
\end{cases}
\;\;\left( \text{respectively} \;\;
T^{-}(x) \coloneqq
\begin{cases}
f_{0}(x) & \text{if} \; 0 \leq x \leq q,\\
f_{1}(x) & \text{if} \; q < x \leq 1,
\end{cases}\right)
\end{align*}
where
\begin{enumerate}
\item $f_{0}: [0, q] \to [0, 1]$ and $f_{1}: [q, 1] \to [0, 1]$ are continuous, strictly increasing, functions, with $f_{0}(0) = 0$ and $f_{1}(1) = 1$ and either $1 > f_{0}(q) > f_{1}(q) \geq 0$ or $1 \geq f_{0}(q) > f_{1}(q) > 0$, and
\item there exists $s > 1$ such that $\lvert f_{i}(x) - f_{i}(y) \rvert \geq s \lvert x - y \rvert$, for $i \in \{ 0, 1\}$ and $x \in [0, 1]$.
\end{enumerate}
A \textit{Lorenz} (\textit{dynamical}) \textit{system} with critical point $q$ is defined to be a dynamical system $([0, 1], T)$, where $T$ is either an upper or lower Lorenz map with critical point $q$.  
\end{definition}

\begin{definition}
A tuple  $(a, p)$ is called \textit{admissible} if it belongs to the set $\{ (z, w) \in (1, 2) \times (0, 1) : 1 - z^{-1} \leq w \leq z^{-1} \}$.  An upper or lower Lorenz map with critical point $p$ is called \textit{uniform} if $(a, p)$ is admissible and if $f'_{0}(x) = a = f'_{1}(y)$, for all $x \in (0, p)$ and $y \in (p, 1)$.  We denote such maps by the symbols $U^{+}_{a, p}$ or $U^{-}_{a, p}$ respectively.  Specifically, the maps $U^{+}_{a, p}$ and $U^{-}_{a, p}$ are given by,
\begin{align*}
U^{+}_{a, p}(x) \coloneqq
\begin{cases}
a x & \text{if} \; 0 \leq x < p,\\
a x + 1 -a & \text{if} \; p \leq x \leq 1,
\end{cases}
\quad
U^{-}_{a, p}(x) \coloneqq
\begin{cases}
a x & \text{if} \; 0 \leq x \leq p,\\
a x + 1 -a & \text{if} \; p < x \leq 1.
\end{cases}
\end{align*}
\end{definition}

Throughout we use the convention that $\pm$ means either $+$ or $-$.  When we write, `given a Lorenz map $T^{\pm}$ with critical point $q$', we require both $T^{+}$ and $T^{-}$ to be defined using the same functions $f_{0}$ and $f_{1}$.  Further, let $\mathbb{N}$ denote the set of positive integers, $\mathbb{N}_{0}$ denote the set of non-negative integers and $\mathbb{R}$ denote the set of real numbers.

We let $\Omega \coloneqq \{0, 1\}^{\infty}$ denote the set of all infinite strings $\omega_{0} \, \omega_{1} \, \omega_{2} \cdots$ consisting of elements of the set $\{0, 1 \}$.  It is well-known that the set $\Omega$ is a complete compact metric space with respect to the metric $d: \Omega \times \Omega \to \mathbb{R}$ given by
\begin{align*}
d(\omega, \sigma) \coloneqq
\begin{cases}
0 & \text{if} \; \omega = \sigma,\\
2^{- \lvert\omega \wedge \sigma\rvert} & \text{otherwise},
\end{cases}
\end{align*}
where $\rvert\omega \wedge \sigma\lvert \coloneqq \min \{ \, n \in \mathbb{N} \, : \, \omega_{n} \neq \sigma_n \}$, for all $\omega \coloneqq \omega_{0} \, \omega_{1} \, \cdots, \sigma \coloneqq \sigma_{0} \, \sigma_{1} \, \cdots \in \Omega$ with $\omega \neq \sigma$.  Throughout we assume that $\Omega$ is equipped with the metric $d$ and is endowed with the lexicographic ordering which will be denoted by the symbols $\succ$ and $\prec$.  

\begin{definition}\label{Defn:Itinerary_map1}
The \textit{upper} (or \textit{lower}) \textit{itinerary}, $\tau_{q}^{+}(x)$ (respectively $\tau^{-}_{q}(x)$) of a point $x \in [0,1]$ under $T^{+}$ (respectively $T^{-}$) with critical point $q$ is the string $\omega_{0} \, \omega_{1} \, \omega_{2} \, \cdots \in \Omega$ (respectively $\sigma_{0} \, \sigma_{1} \, \sigma_{2}\, \cdots \in \Omega$), where
\begin{align*}
 \omega_{k} \coloneqq \begin{cases}
0 & \text{if} \; (T^{+})^{k}(x) < q\\
1 & \text{if} \; (T^{+})^{k}(x) \geq q.
\end{cases}
\quad\left(\text{respectively} \quad
\sigma_{k} \coloneqq \begin{cases}
0 & \text{if} \; (T^{-})^{k}(x) \leq q\\
1 & \text{if} \; (T^{-})^{k}(x) > q.
\end{cases}
\right).
\end{align*}
To distinguish the itinerary map of a uniform Lorenz map $U_{a,p}^{\pm}$ we use the symbol $\mu_{a, p}^{\pm}$.
\end{definition}

Let $(T^{+})^{n}$ denote the $n$-fold composition of $T^{+}$ with itself, where $(T^{+})^{0}(x) \coloneqq x$ for $x\in [0,1]$ and $n \in \mathbb{N}$.

\begin{definition}\label{Defn:Itinerary_map3}
Given a Lorenz map $T^{\pm}: [0, 1] \circlearrowleft$ with critical point $q$, we let $\Omega_{q}^{\pm} \subset \Omega$ denote the image of the unit interval $[0,1]$ under the mapping $\tau_{q}^{\pm}$.   The set $\Omega_{q}^{\pm}$ is called the \textit{address space} of the dynamical system $([0, 1], T^{\pm})$.  To distinguish the address space of a uniform Lorenz system $([0, 1], U_{a, p}^{\pm})$, we we use the symbol $\Omega_{a, p}^{\pm}$.
\end{definition}

Given a Lorenz map $T^{\pm}$, we let $h(T^{\pm})$ denote its topological entropy, which we will define  in \Cref{Section2}.  Since $h(T^{+}) = h(T^{-})$, we let $h(T)$ denote this common value; see \Cref{RMK:+=-}.

Finally, let $g_{0, a}(x) \coloneqq x/a$ and $g_{1,a}(x) \coloneqq x/a + (1-a^{-1})$, for each $a \in (1, 2)$ and $x \in [0, 1]$.  The \textit{coding map} $\pi_{a}: \Omega \to [0, 1]$ is defined by
\begin{align*}\label{EQ:Projection_Map}
\pi_a (\omega_{0} \, \omega_{1} \, \omega_{2} \cdots) \coloneqq \lim_{n \to \infty} g_{\omega_{0}, a} \circ g_{\omega_{1}, a} \circ \dots \circ g_{\omega_{n}, a}(1) = \left(1- a^{-1}\right) \sum_{k=0}^{\infty} \, \omega_{k} \, a^{-k}.
\end{align*}
With the above we can now state our main results. For ease of notation we let $\alpha \coloneqq \tau_{q}^{-}(q)$ and $\beta \coloneqq \tau_{q}^{+}(q)$. 

\begin{theorem}\label{Main_THM}
Let $([0,1], T^{\pm})$ denote a Lorenz system with critical point $q$ such that $T^{-}(q) \neq 1$ and $T^{+}(q) \neq 0$. Then the following statements are equivalent for each $a \in \mathbb{R}$.
\begin{enumerate}
\item The value $a$ belongs to the open interval $(\exp(h(T)), 2)$.
\item The open interval $(\pi_{a}(\alpha),\pi_{a}(\beta)) \cap (1-a^{-1}, a^{-1})$ is non-empty and $\alpha \prec \mu_{a, p}^{-}(p) \prec \mu_{a, p}^{+}(p) \prec \beta$, for all $p\in (\pi_{a}(\alpha),\pi_{a}(\beta)) \cap (1-a^{-1}, a^{-1})$.
\item The open interval $(\pi_{a}(\alpha),\pi_{a}(\beta)) \cap (1-a^{-1}, a^{-1})$ is non-empty and $\Omega_{q}^{-} \subset \Omega_{a, p}^{-}$ and $\Omega_{q}^{+} \subset \Omega_{a, p}^{+}$, for all $p \in (\pi_{a}(\alpha),\pi_{a}(\beta)) \cap (1-a^{-1}, a^{-1})$.
\end{enumerate}
\end{theorem}

%\newpageEQ:END1

\begin{theorem}\label{THM:END}
Let $([0,1], T^{\pm})$ denote a Lorenz system with critical point $q$.
\begin{enumerate}
\item If $T^{-}(q) = 1$, then the following are equivalent
\begin{enumerate}
\item $a \in (\exp(h(T)), 2)$.
\item There exists a unique $p \in [1-a^{-1}, a^{-1}]$, given by $p = a^{-1}$, so that $\alpha = \mu_{a, a^{-1}}^{-}(a^{-1}) \prec \mu_{a, a^{-1}}^{+}(a^{-1}) \prec \beta$.
\item There exists a unique $p \in [1-a^{-1}, a^{-1}]$, given by $p = a^{-1}$, so that $\Omega_{q}^{-} \subset \Omega_{a, p}^{-}$ and $\Omega_{q}^{+} \subset \Omega_{a, p}^{+}$.
\end{enumerate}
\item If $T^{+}(q) = 0$, then the following are equivalent
\begin{enumerate}
\item $a \in (\exp(h(T)), 2)$.
\item There exists a unique $p \in [1-a^{-1}, a^{-1}]$, given by $p = 1 - a^{-1}$, so that $\alpha \prec \mu_{a, a^{-1}}^{-}(a^{-1}) \prec \mu_{a, a^{-1}}^{+}(a^{-1}) = \beta$.
\item There exists a unique $p \in [1-a^{-1}, a^{-1}]$, given by $p = 1 - a^{-1}$, so that $\Omega_{q}^{-} \subset \Omega_{a, p}^{-}$ and $\Omega_{q}^{+} \subset \Omega_{a, p}^{+}$.
\end{enumerate}
\end{enumerate}
\end{theorem} 

\begin{remark}
In \Cref{Main_THM} it is necessary to take the intersection of the intervals $(\pi_{a}(\alpha),\pi_{a}(\beta))$ and $(1-a^{-1}, a^{-1})$ instead of only the interval $(\pi_{a}(\alpha),\pi_{a}(\beta))$. Otherwise the inequality $\pi_{a}(\alpha) < 1- a^{-1}$ or $\pi_{a}(\beta) > a^{-1}$ may occur, and so, the corresponding uniform Lorenz system will not be well defined; see \Cref{EXMP:inequalities}.
\end{remark}

\begin{remark}
For each $a > \exp(h(T))$, \Cref{Main_THM,THM:END} fully classify the points $p$ belonging to the interval $[1-a^{-1}, a^{-1}]$, such that either $\tau_{q}^{-}(q) \preceq \mu_{a, p}^{-}(p) \prec \mu_{a, p}^{+}(p) \prec \tau_{q}^{+}(q)$ or $\tau_{q}^{-}(q) \prec \mu_{a, p}^{-}(p) \prec \mu_{a, p}^{+}(p) \preceq \tau_{q}^{+}(q)$ hold, which, as we will see, implies  an embedding of address spaces, or formally, $\Omega_{q}^{-} \subset \Omega_{a, p}^{-}$ and $\Omega_{q}^{+} \subset \Omega_{a, p}^{+}$.  
\end{remark}

In the final section of this paper we present a new algorithm, based on \Cref{Main_THM,THM:END}, which calculates the topological entropy of a Lorenz map.  The main idea behind the algorithm is the following.  The algorithm first uses an efficient method to calculate the address spaces of a given Lorenz system $([0, 1], T)$.  Then, in a systematic way, it compares the address spaces of $([0, 1], T)$ to the address spaces of a subclass of the family of uniform Lorenz systems.  By a well-known result of Parry \cite{Par} the topological entropy of each member of this subclass of systems is known.  Using \Cref{Main_THM,THM:END} the algorithm is then able to obtain an estimate of the topological entropy of the given system.

\subsection{Outline}

 \Cref{sec:pre} contains  necessary preliminaries.  The concepts of topological entropy and topological (semi-) conjugacy are introduced in \Cref{Section2};  properties of itinerary maps are presented in \Cref{Sec:Kneading}; and several required auxiliary results are proved in \Cref{Sec:Prelim}. \Cref{Sec:Proof_of_Main_THM} contains the proofs of \Cref{Main_THM,THM:END}.  We conclude with \Cref{Sec:Algorithm}, where the statement and a proof of validity of a new algorithm for computing the topological entropy of a Lorenz (dynamical) system is given.

\section{Preliminaries}\label{sec:pre}

In this section, various auxiliary results are proved in preparation for the proof of \Cref{Main_THM,THM:END}.

\subsection{Entropy and topological conjugacy}\label{Section2}

Recall the definition of topological entropy and topological (semi-) conjugacy.

\begin{definition}\label{defEntropy}
Let $T^{\pm}$ be a Lorenz map with critical point $q$.  For $\omega \in \Omega$, the string consisting of the first $n \in \mathbb{N}$ symbols of $\omega$ is denoted by $\omega\vert_{n}$ and $\omega\vert_{0}$ denotes the empty word.  We set $\Omega_{q, n}^{\pm} \coloneqq \{ \omega\vert_{n} : \omega \in \Omega_{q}^{\pm} \}$ and let $\lvert \Omega_{q, n}^{\pm} \rvert$ denote the cardinality of the set $\Omega_{q, n}^{\pm}$, for each $n \in \mathbb{N}$.  The \textit{topological entropy} $h(T^{\pm})$ of $([0, 1], T^{\pm})$ is defined by $\displaystyle h(T^{\pm}) \coloneqq \lim_{n \to \infty}  \ln ( \lvert \Omega_{q, n}^{\pm} \rvert^{1/n} )$.
\end{definition}

\begin{remark}\label{RMK:+=-}
It is well-known  that $h(T^{+}) = h(T^{-}) \leq \ln(2)$.  Thus, for ease of notation,  we denote the common value $h(T^{+}) = h(T^{-})$ by $h(T)$ .
\end{remark}

\begin{theorem}\cite{Par,HR}\label{thm:uniform}
If $(a, p)$ is an admissible pair, then $h\left(U^{+}_{a, p}\right) = h\left(U^{-}_{a, p}\right)$.  Moreover, this common value is equal to $\ln (a)$.
\end{theorem}

\begin{definition}\label{Defn:conj}
Two maps $R: X \circlearrowleft$ and $S: Y \circlearrowleft$ defined on compact metric spaces are called \textit{topologically conjugate} if there exists a homeomorphism $\hbar: X \to Y$ such that $S \circ \hbar = \hbar \circ R$.  If $\hbar$ is continuous and surjective then $R$ and $S$ are called  \textit{semi-conjugate}.
\end{definition}

When we write, `two dynamical systems are topologically (semi-) conjugate', we mean that the associated maps are topologically (semi-) conjugate.  

\begin{lemma}[\cite{G}]\label{lemma1}
\leavevmode
\begin{enumerate}
\item If two Lorenz systems $([0,1], T^{\pm})$ and $([0,1], R^{\pm})$ are topologically conjugate, then the address spaces are equal and hence, $h(T) = h(R)$. 
\item If a Lorenz system $([0,1], T^{\pm})$ with critical point $q$ is semi-conjugate to a Lorenz system $([0,1], R^{\pm})$ with critical point $p$, then $\Omega_{p}^{\pm}\subseteq \Omega_{p}^{\pm}$ and $h(T) = h(R)$.
\end{enumerate}
\end{lemma}

\subsection{Properties of itinerary maps}\label{Sec:Kneading}

We next state properties of the itinerary maps  $\mu_{a, p}^{\pm}$ of uniform Lorenz systems.  Throughout this section $(a, p)$ will denote an admissible pair.

\begin{lemma}[\cite{BHV}]\label{lem2}
\leavevmode
\begin{enumerate}
\item The map $[0, 1] \ni x \mapsto \mu_{a, p}^{+}(x)$ is strictly increasing and right-continuous. Moreover, for all $x \in (0, 1)$, we have that $\displaystyle \mu_{a, p}^{-}(x) = \lim_{\epsilon \searrow 0} \mu_{a, p}^{+}(x - \epsilon)$.
\item The map $[0, 1] \ni x \mapsto \mu_{a, p}^{-}(x)$ is strictly increasing and left-continuous.  Moreover, for all $x \in (0, 1)$, we have that $\displaystyle \mu_{a, p}^{+}(x) = \lim_{\epsilon \searrow 0} \mu_{a, p}^{-}(x + \epsilon)$.
\item The map $p \mapsto \mu_{a, p}^{+}(p)$ is strictly increasing and right-continuous.
\item The map $p \mapsto \mu_{a, p}^{-}(p)$ is strictly increasing and left-continuous.
\end{enumerate}
\end{lemma}

Finally, we conclude with the a result which links the coding map $\pi_{a}$, defined in \eqref{EQ:Projection_Map}, and the itinerary maps $\mu_{a, p}^{\pm}$.  This requires the following definition.

\begin{definition}\label{DEFN:SHIFT_INV}
The continuous map $S: \Omega \circlearrowleft$ defined by $S(\omega_{0} \, \omega_{1} \, \omega_{2} \cdots) \coloneqq \omega_{1} \, \omega_{2} \, \omega_{3} \cdots,$ is called the \textit{shift map} and a subset $\Lambda$ of $\Omega$ is called \textit{forward shift sub-invariant} if $S(\Lambda) \subseteq \Lambda$.
\end{definition}

\begin{proposition}\label{prop:code-map}
We have that  $\pi_{a} \left ( \mu_{a, p}^{\pm}(x)\right ) = x$, for all $x \in [0,1]$, and that the following diagram commutes
\[
\begin{array}
[c]{ccc}
\Omega_{a, p}^{\pm} & \overset{S}{\longrightarrow} & \Omega_{a, p}^{\pm}\\
& & \\
\pi_a \downarrow\text{\ \ \ \ } &  & \text{ \ \ \ }\downarrow\pi_a \\ & & \\
 \lbrack 0,1 \rbrack & \underset{U_{a, p}^{\pm}} {\longrightarrow} & \lbrack 0,1\rbrack.
\end{array}
\]
\end{proposition}

\begin{proof}
The result is readily verifiable from the definitions of the maps involved.   Also a sketch of the proof of the result appears in \cite[Section~5]{BHV} and \cite[Section~2.2]{GH}.
\end{proof}

\subsection{Auxiliary results}\label{Sec:Prelim}

In the following auxiliary results used in the proofs of \Cref{Main_THM,THM:END}, let $([0,1], T^{\pm})$ denote a Lorenz system with critical point $q$, let $\tau_{q}^{\pm}$ denote the associated itinerary map, and let $\Omega_{q}^{\pm}$ denote the associated address space.

\begin{lemma}\label{lem:shift}
The address space $\Omega_{q}^{\pm}$ is forward shift sub-invariant.
% that is $S(\Omega_{q}^{\pm}) \subseteq \Omega_{q}^{\pm}$.
\end{lemma}

\begin{proof}
This is a direct consequence of \Cref{prop:code-map}.
\end{proof}

A partial version of the following result can be found in \cite[Lemma~1]{H}. However, to the best of our knowledge, Theorem~\ref{thm:Main1.5} first appeared in \cite[Theorem~1]{HS}. 

\begin{definition}
The strings $\alpha \coloneqq \tau_{q}^{-}(q)$ and $\beta \coloneqq \tau_{q}^{+}(q)$ are called the \textit{critical itineraries}.  
\end{definition}

\begin{theorem}\label{thm:Main1.5}
The spaces $\Omega_{q}^{+}$ and $\Omega_{q}^{-}$ are uniquely determined by $\alpha$ and $\beta$ as follows:
\begin{align*}
\Omega_{q}^{+}  &= \{ \omega \in \Omega : S^{n}(\omega) \prec \alpha \; \text{or} \; \beta \preceq S^{n}(\omega), \; \text{for all} \; n \in \mathbb{N}_{0} \},\\
\Omega_{q}^{-}  &= \{ \omega \in \Omega :  S^{n} (\omega) \preceq \alpha \; \text{or} \; \beta \prec S^{n}(\omega), \, \; \text{for all} \; n \in \mathbb{N}_{0} \}.
\end{align*}
\end{theorem}

\begin{corollary}\label{cor:cor2}
Let $a \in (1,2)$ be fixed. 
\begin{enumerate}
\item  If there exists $p$ such that $(a, p)$ is admissible and $\alpha \preceq \mu_{a, p}^{-}(p) \prec \mu_{a, p}^{+}(p) \preceq \beta$, then  $h(T) \leq \ln (a)$.  
\item  If there exists $p$ such that $(a, p)$ is admissible and $\mu_{a, p}^{-}(p)  \preceq \alpha  \prec  \beta \preceq \mu_{a, p}^{+}(p)$, then   $h(T) \geq \ln (a)$.
\end{enumerate}
\end{corollary}

\begin{proof}
This is a direct consequence of \Cref{defEntropy} and \Cref{thm:uniform,thm:Main1.5}.
\end{proof}

In the proofs of some of the following results we let $\overline{0}$ denote the element $0 \, 0 \, \cdots \in \Omega$ and $\overline{1}$ the element $1 \,1  \, \cdots \in \Omega$,

\begin{lemma}\label{thm:Main2}
Given $a \in (1,2)$, there exists $p$ such that $(a, p)$ is admissible and either
\begin{subequations}
\begin{equation}
\alpha \preceq \mu_{a, p}^{-}(p) \prec \mu_{a, p}^{+}(p) \preceq \beta \label{comeq1}
\end{equation}
or
\begin{equation}
\mu_{a, p}^{-}(p)  \preceq \alpha  \prec  \beta \preceq \mu_{a, p}^{+}(p) \label{comeq2}. 
\end{equation}
\end{subequations}
Hence, in the first case $h(T) \leq \ln (a)$, and in the second case $h(T) \geq \ln (a)$.
\end{lemma}

\begin{proof}
Since a lower itinerary starts with $0$ and an upper itinerary starts with $1$, we have $\alpha \preceq 0 \overline{1} = \mu^{-}_{a, a^{-1}}(a^{-1})$ and $\mu^{+}_{a,1 - a^{-1}}(1 - a^{-1}) = 1 \overline{0} \preceq \beta$.  Hence, the inequalities given in \eqref{comeq1} hold for $p = 1 - a^{-1}$, unless 
\begin{align}\label{eq1}
\mu^{-}_{a, 1-a^{-1}}(1 - a^{-1}) \prec \alpha,
\end{align}
and, similarly, the inequalities given in \eqref{comeq1} hold for $p = a^{-1}$, unless
\begin{align} \label{eq2}
\mu^{+}_{a, a^{-1}}(a^{-1}) \succ \beta.
\end{align}
If the inequalities given in \eqref{comeq1} are false for both $p = 1-a^{-1}$ and $p = a^{-1}$, then the inequalities of both \eqref{eq1} and \eqref{eq2} hold.  Let $p_{1} \coloneqq \sup \{ p : \mu^{-}_{a, p}(p) \preceq  \alpha \; \text{and} \; \mu^{+}_{a, p}(p)  \preceq \beta \}$ and  $p_{2} \coloneqq \inf \{ p : \mu^{-}_{a, p}(p) \succeq  \alpha \; \text{and} \; \mu^{+}_{a, p}(p)  \succeq \beta \}$.  \Cref{lem2} implies that $p_{2} \geq p_{1}$ and that if $p_{2} > p > p_{1}$, then either the inequalities given in \eqref{comeq1} or the inequalities given in \eqref{comeq2} hold for $p$.  If $p_{1} = p_{2}$, then \Cref{lem2} implies that the inequalities given in \eqref{comeq2} hold at $p = p_{1} = p_{2}$.

The remaining assertion follows from \Cref{cor:cor2}.
\end{proof}

\begin{lemma}\label{lem:1234}
Let $a \in (\exp(h(T)), 2)$ be fixed.  If $T^{-}(q) \neq 1$ and $T^{+}(q) \neq 0$, then there exists a non-empty open interval $V \subseteq [1-a^{-1}, a^{-1}]$, such that $\alpha \prec \mu_{a, t}^{-}(t) \prec \mu_{a, t}^{+}(t) \prec \beta$, for all $t \in V$.  Moreover, letting
\begin{subequations}\label{eq:pa}
\begin{equation}
p_{1} (a) \!\coloneqq\! \max \left\{1\!-\!a^{-1}, \sup \left\{ p \in [1\!-\!a^{-1}, a^{-1}] : \mu^{-}_{a, p}(p) \preceq  \alpha \, \text{and} \, \mu^{+}_{a, p}(p)  \preceq \beta \right\} \right\}
\end{equation}
and
\begin{equation}
p_{2} (a) \!\coloneqq\! \min \left\{ a^{-1}, \inf \left\{ p \in [1\!-\!a^{-1}, a^{-1}] : \mu^{-}_{a, p}(p) \succeq \alpha \, \text{and} \, \mu^{+}_{a, p}(p)  \succeq \beta \right\}\right\},
\end{equation}
\end{subequations}
we have that $V \subseteq (p_1 (a), p_2 (a))$ and hence $p_1 (a) < p_2 (a)$.
\end{lemma}

\begin{proof}
Since $\ln (a) > h(T)$, by \Cref{thm:Main2}, there exists $p$ such that $(a, p)$ is admissible and that least one of the following sets of inequalities hold:
\begin{subequations}
\begin{equation}
\alpha \prec \mu_{a, p}^{-}(p) \prec \mu_{a, p}^{+}(p) \preceq \beta,\label{comeq3}
\end{equation}
or
\begin{equation}
\alpha \preceq \mu_{a, p}^{-}(p) \prec \mu_{a, p}^{+}(p) \prec \beta.\label{comeq4}
\end{equation}
\end{subequations}
(Observe that the situation in which $\alpha = \mu_{a, p}^{-}(p)$ and $\mu_{a, p}^{+}(p) = \beta$ cannot occur since $\ln (a) > h(T)$.)  Let such a $p$ be fixed.  If $p = 1 - a^{-1}$, then, by the definition of the itinerary map and the fact that  $T^{+}(q) \neq 0$, we have that $\beta \succ 1\overline{0}$ and that $\mu_{a, p}^{+}(p) = 1\overline{0}$.  Hence, the inequalities given in \eqref{comeq4} hold.  Similarly, if $p = a^{-1}$, then $\alpha \prec 0\overline{1}$ and $\mu_{a, p}^{-}(p) = 0\overline{1}$, hence the inequalities given in \eqref{comeq3} hold.

Suppose that $p \not\in \{ 1-a^{-1}, a^{-1}\}$ and that the inequalities given in \eqref{comeq3} hold.  Let $r \coloneqq d(\mu_{a, p}^{-}(p), \alpha) > 0$.  By \Cref{lem2}~(ii), we have $\lim_{\epsilon \searrow 0} d(\mu_{a, p-\epsilon}^{-}(p - \epsilon), \mu_{a, p}^{-}(p)) = 0$.  Therefore, there exists $\delta = \delta(r) \in (0, p - 1 + a^{-1})$ such that, for all $\epsilon < \delta = \delta(r)$,  $d(\mu_{a, p-\epsilon}^{-}(p-\epsilon), \mu_{a, p}^{-}(p)) < r/2$.  Now, \Cref{lem2}~(iv), the definition of the metric $d$ and that of the lexicographic ordering, together with the above inequality, imply that $\alpha \prec \mu_{a, p-\epsilon}^{-}(p-\epsilon) \prec \mu_{a, p}^{-}(p)$, for all $\epsilon < \delta$.  Thus, by \Cref{lem2}~(iii), we have that $\mu_{a, p-\epsilon}^{+}(p-\epsilon) \prec \mu_{a, p}^{+}(p)$, for all $\epsilon < \delta$.  Therefore, by the definition of the itinerary maps $t \mapsto \mu_{a, t}^{\pm}(t)$ and by the assumption that the inequalities given in \eqref{comeq3} hold, we have that $\alpha \prec \mu_{a, p-\epsilon}^{-}(p-\epsilon) \prec \mu_{a, p-\epsilon}^{+}(p - \epsilon) \prec  \mu_{a, p}^{+}(p ) \prec \beta$, for all $\epsilon < \delta$.  Furthermore, since $\delta \in (0, p - 1 + a^{-1})$ and since $p \in (1-a^{-1}, a^{-1}]$, it follows that $(p - \delta, p) \subset (1-a^{-1}, a^{-1})$.  Setting $V \coloneqq (p - \delta, p)$ yields the required result.

A similar argument yields the required result under the assumption of the inequalities given in \eqref{comeq4} for our fixed $p$.

The remaining assertion is an immediate consequence of the definitions of $p_{1} (a)$ and $p_{2} (a)$ and \Cref{lem2}.
\end{proof}

\begin{lemma}\label{lem:bounds1.0}
The restriction of the coding map $\pi_{a}$ to the set $\Omega_{a, p}^{+}$ and the restriction of $\pi_{a}$ to the set $\Omega_{a, p}^{-}$ are strictly increasing, for all admissible pairs $(a, p)$. Furthermore, the restriction of the coding map $\pi_{a}$ to the set $\Omega_{a, p}^{+} \cup \Omega_{a, p}^{-}$ is increasing.  
\end{lemma}

\begin{proof}
The first statement follows from \Cref{lem2} and \Cref{prop:code-map}.

To show that the restriction of $\pi_{a}$ to $\Omega_{a, p}^{+} \cup \Omega_{a, p}^{-}$ is increasing, let $\omega, \omega' \in \Omega_{a, p}^{+} \cup \Omega_{a, p}^{-}$ be such that $\omega \preceq \omega'$.  One of the following situations must now occur.
\begin{enumerate}
\item  $\omega, \omega' \in \Omega^{+}_{a, p}\;$ or $\;\omega, \omega' \in \Omega^{-}_{a, p}$,
\item  $\omega \in \Omega^{-}_{a, p} \setminus \Omega^{+}_{a, p}$ and $\omega' \in \Omega^{+}_{a, p} \setminus \Omega^{-}_{a, p}$, or 
\item $\omega \in \Omega^{+}_{a, p} \setminus \Omega^{-}_{a, p}$ and $\omega' \in \Omega^{-}_{a, p} \setminus \Omega^{+}_{a, p}$.  
\end{enumerate}
If (i) occurs, then by the fact that the restriction of $\pi_{a}$ to the set $\Omega_{a, p}^{+}$ is strictly increasing  and the restriction of $\pi_{a}$ to the set $\Omega_{a, p}^{-}$ is strictly increasing, it follows that $\pi_{a}(\omega) < \pi_{a}(\omega')$.

Suppose (ii) occurs.  Let $y \coloneqq \pi_{a}(\omega)$ and $z \coloneqq \pi_{a}(\omega')$. By way of contradiction, assume  $y > z$.  \Cref{lem2} implies
\begin{align}\label{eq:lim_relation_+_}
\mu_{a, p}^{+}(z) &= \lim_{\epsilon \searrow 0} \mu_{a, p}^{-}(z + \epsilon).
\intertext{Now}
\omega' = \mu^{+}_{a, p}(z) &= \lim_{\epsilon \searrow 0} \mu^{-}_{a, p}(z + \epsilon) \prec \mu^{-}_{a, p}(y) = \omega,
\label{eq:3.13}
\end{align}
where the first equality holds since  $\omega'\in \Omega_{a, p}^{+}$, and so there exists $x \in [0,1]$ such that $ \omega ' =  \mu^{+}_{a, p}(x)$.  Then, by \Cref{prop:code-map}, we have $z \coloneqq \pi_{a}(\omega') = \pi_{a}(\mu^{+}_{a, p}(x)) = x$.  Hence $\omega' = \mu^{+}_{a, p}(x) =\mu^{+}_{a, p}(z)$.  The second equality in \eqref{eq:3.13} follows from \eqref{eq:lim_relation_+_}; the following inequality is due to \Cref{lem2} and the fact that $y > z + \epsilon$ for all sufficiently small $\epsilon > 0$; and the last equality follows in exactly the same way as the first equality. Therefore, $\omega' \prec \omega$, which contradicts our hypothesis, namely that $\omega \preceq \omega'$.

If (iii) occurs, then similar argument to those given above will yield that \mbox{$\pi_{a}(\omega) \leq \pi_{a}(\omega')$}.
\end{proof}

\begin{lemma}\label{lem:bounds}
If $2 > a > \exp(h(T))$, $T^{-}(q) \neq 1$ and $T^{+}(q)\neq 0$, then $\pi_{a}(\alpha) < \pi_{a}(\beta)$ and
\begin{align}
\emptyset \neq (p_1 (a), p_2 (a)) \subseteq (\pi_{a}(\alpha),\pi_{a}(\beta)) \cap (1-a^{-1}, a^{-1}),
\end{align}
where $p_{1}(a)$ and $p_{2}(a)$ are the real numbers defined in \eqref{eq:pa} respectively.
\end{lemma}

\begin{proof}
Suppose that $a \in (\exp(h(T)), 2)$. By \Cref{thm:Main2}, there exists $p$ such that $(a, p)$ is admissible and either one of the following sets of inequalities hold,
\begin{enumerate}
\item $\alpha \prec \mu_{a, p}^{-}(p)$ and $\mu_{a, p}^{+}(p) \preceq \beta$, or
\item $\alpha \preceq \mu_{a, p}^{-}(p)$ and $\mu_{a, p}^{+}(p) \prec \beta$.
\end{enumerate}
Note that the situation where $\alpha = \mu_{a, p}^{-}(p)$ and $\mu_{a, p}^{+}(p) = \beta$ cannot occur as $a > \exp(h(T))$.

Assume that (i) occurs.  By \Cref{thm:Main1.5} it follows that $\Omega_{q}^{-} \subset \Omega_{a, p}^{-}$ and $\Omega_{q}^{+} \subseteq \Omega_{a, p}^{+}$.  In particular, $\alpha \in \Omega^{-}_{a, p}$ and $\beta \in \Omega^{+}_{a, p}$.  Since, by \Cref{lem:bounds1.0}, the coding map $\pi_{a}$ is strictly increasing on $\Omega_{a, p}^{+}$ and on $\Omega_{a, p}^{-}$, we have
\begin{align}\label{eq:bounds1}
\pi_{a}(\alpha) < \pi_{a}(\mu_{a, p}^{-}(p)) = p = \pi_{a}(\mu_{a, p}^{+}(p)) \leq \pi_{a}(\beta).
\end{align}
If (ii) occurs, then essentially the same arguments as those above yield
\begin{align}\label{eq:bounds2}
\pi_{a}(\alpha) \leq \pi_{a}(\mu_{a, p}^{-}(p)) = p = \pi_{a}(\mu_{a, p}^{+}(p)) < \pi_{a}(\beta).
\end{align}
Hence, $\pi_{a}(\alpha) < \pi_{a}(\beta)$ and $[\pi_{a}(\alpha),\pi_{a}(\beta)] \cap [1-a^{-1}, a^{-1}] \neq \emptyset$.

We now show that the open interval $(\pi_{a}(\alpha),\pi_{a}(\beta)) \cap (1-a^{-1}, a^{-1})$ is non-empty.  Observe that, by \Cref{lem2} and the definition of $p_1 (a)$ and $p_2 (a)$, for all $t \in (p_1 (a), p_2 (a))$, there are two possible sets of inequalities that can occur:
\begin{enumerate}
\item[(a)] $\alpha \succ \mu^{-}_{a, t}(t)$ and $\beta \prec \mu^{+}_{a, p}(t)$, or
\item[(b)] $\alpha \prec \mu^{-}_{a, t}(t)$ and $\beta \succ \mu^{+}_{a, p}(t)$.
\end{enumerate}
The set of inequalities in (a), however, cannot occur.  If they did, \Cref{thm:uniform,thm:Main1.5} and the definition of topological entropy, would imply $\ln (a) \leq h(T)$, contradicting our hypothesis.  Thus, by \eqref{eq:bounds1} and \eqref{eq:bounds2} we have
\begin{align}\label{eq:bounds-p_1-p_2}
(p_1 (a), p_2 (a)) \subseteq [\pi_{a}(\alpha), \pi_{a}(\beta)] \cap [1-a^{-1}, a^{-1}].
\end{align}
Since our hypothesis is the same as that of \Cref{lem:1234}, we have that $p_1 (a) < p_2 (a)$, and so, the open interval $(p_1 (a), p_2 (a))$ is non-empty.  This, in tandem with \eqref{eq:bounds-p_1-p_2}, implies $(\pi_{a}(\alpha), \pi_{a}(\beta)) \cap (1-a^{-1}, a^{-1}) \neq \emptyset$.
\end{proof}

\section{Proof of \Cref{Main_THM,THM:END}}\label{Sec:Proof_of_Main_THM}

%With the above preparation we are now able to prove \Cref{Main_THM,THM:END}.

\begin{proof}[Proof of \Cref{Main_THM}.]
We proceed by showing that (i) $\Rightarrow$ (ii) $\Rightarrow$ (iii) $\Rightarrow$ (i).

(i) $\Rightarrow$ (ii) Fix $a \in (\exp(h(T)), 2)$.  By \Cref{lem:bounds}, we have $\emptyset\neq(p_1 (a), p_2 (a)) \subseteq (\pi_{a}(\alpha), \pi_{a}(\beta)) \cap (1-a^{-1}, a^{-1})$.  Moreover, for each $p \in (p_1 (a), p_2 (a)) \subseteq (\pi_{a}(\alpha), \pi_{a}(\beta)) \cap (1-a^{-1}, a^{-1})$,
\begin{align}\label{eq:not_equal_comparasions}
\alpha \prec \mu_{a, p}^{-}(p) \quad \text{and} \quad \mu_{a, p}^{+}(p) \prec \beta.
\end{align}
(We remind the reader that $\alpha \coloneqq \tau_{q}^{-}(q)$ and $\beta \coloneqq \tau_{q}^{+}(q)$ are the critical itineraries of $([0, 1], T^{\pm})$.)  Let such a $p$ be fixed.  By \Cref{thm:Main1.5} and the inequalities given in \eqref{eq:not_equal_comparasions} we have 
\begin{align}\label{eq:subset_code_space_p_q}
\Omega_{q}^{-} \subset \Omega_{a, p}^{-} \quad \text{and} \quad \Omega_{q}^{+} \subset \Omega_{a, p}^{+}.
\end{align}
By \Cref{thm:Main1.5}, the inclusions in \eqref{eq:subset_code_space_p_q}, and the fact that the map $\pi_{a}\vert_{\Omega_{a, p}^{+} \cup \Omega^{-}_{a, p}}$ is increasing (\Cref{lem:bounds1.0}), we have that $\pi_{a}(\omega) \in [0, \pi_{a}(\alpha)] \cup [\pi_{a}(\beta), 1]$, for all $\omega \in \Omega_{q}^{+} \cup\Omega_{q}^{-}$.  In other words
\begin{align}\label{eq:codevinterval}
\pi_{a}(\Omega_{q}^{+} \cup\Omega_{q}^{-}) \subseteq  [0, \pi_{a}(\alpha)] \cup [\pi_{a}(\beta), 1].
\end{align}
We claim that, for each $x \in \pi_{a}(\Omega_{q}^{+} \cup \Omega_{q}^{-})$ and $p' \in (\pi_{a}(\alpha),\pi_{a}(\beta)) \cap (1-a^{-1}, a^{-1})$,
\begin{align*}
U_{a,p}^{\pm}(x) = U_{a,p'}^{\pm}(x),
\quad
\text{and}
\quad
U_{a, p'}^{\pm}(\pi_{a}(\Omega_{q}^{+} \cup \Omega_{q}^{-}) ) \subseteq \pi_{a}(\Omega_{q}^{+} \cup \Omega_{q}^{-}), 
\end{align*}
It follows from this claim that, for all  $p' \in (\pi_{a}(\alpha),\pi_{a}(\beta)) \cap (1-a^{-1}, a^{-1})$,
\begin{align}\label{eq:same_itinerary}
 \mu_{a, p'}^{\pm}(x) = \mu_{a, p}^{\pm}(x) \quad \text{for all} \quad x \in \pi_{a}(\Omega_{q}^{+} \cup \Omega_{q}^{-}).
\end{align} 
To prove the claim, let $p' \in (\pi_{a}(\alpha),\pi_{a}(\beta)) \cap (1-a^{-1}, a^{-1})$ and $x \in \pi_{a}(\Omega_{q}^{+} \cup \Omega_{q}^{-})$.  In light of the inclusion given in \eqref{eq:codevinterval} there are two cases, either $x \in \pi_{a}(\Omega_{q}^{+} \cup \Omega_{q}^{-}) \cap [ 0, \pi_{a}(\alpha)]$ or $x \in \pi_{a}(\Omega_{q}^{+} \cup \Omega_{q}^{-}) \cap [\pi_{a}(\beta),  1]$.  As the proofs are essentially the same, we take  $x \in \pi_{a}(\Omega_{q}^{+} \cup \Omega_{q}^{-}) \cap [ 0, \pi_{a}(\alpha)]$.  Since $p, p' \in (\pi_{a}(\alpha),\pi_{a}(\beta)) \cap (1-a^{-1}, a^{-1})$, we have that $\pi_{a}(\alpha) < \min\{ p, p'\}$.
Moreover, $x \leq \pi_{a}(\alpha) < \min \{ p, p' \}$; in particular $x \neq p$ and $x \neq p'$. 
From this and the definition of the functions  $U_{a, p}^{\pm}$, it can be concluded that 
\begin{align}\label{eq:conclude}
U_{a, p}^{\pm}(x) = U_{a, p'}^{\pm}(x).
\end{align}
Since $x \in \pi_{a}(\Omega_{q}^{+} \cup \Omega_{q}^{-}) \cap [0, \pi_{a}(\alpha)]$, there exists $\omega \in\Omega_{q}^{+} \cup \Omega_{q}^{-}$ such that $x=\pi_{a}(\omega)$, and so
\begin{align*}
U_{a, p'}^{\pm}(x) = U_{a, p}^{\pm}(x) = U_{a, p}^{\pm}(\pi_{a}(\omega)) = \pi_{a}(S(\omega)) \in \pi_{a}(\Omega_{q}^{+} \cup \Omega_{q}^{-}),
\end{align*}
where the first equality follows from \eqref{eq:conclude}; the second equality follows from the fact that $x = \pi_{a}(\omega)$; the final equality follows from the inclusions given in \eqref{eq:subset_code_space_p_q} and \Cref{prop:code-map}; and the inclusion $\pi_{a}(S(\omega)) \in \pi_{a}(\Omega_{q}^{+} \cup \Omega_{q}^{-})$  is due to that fact that $\Omega_{q}^{+} \cup \Omega_{q}^{-}$ is forward shift sub-invariant (\Cref{lem:shift}).  Thus the claim is proved. 

By the inclusion given in \eqref{eq:subset_code_space_p_q} we have that $\alpha \in \Omega^{-}_{a, p}$ and $\beta \in \Omega^{+}_{a, p}$.  So there exist $x, y \in [0, 1]$ such that $\alpha = \mu_{a, p}^{-}(x)$ and $\beta = \mu_{a, p}^{+}(y)$.  Therefore, by \Cref{prop:code-map} we have that $\mu_{a, p}^{-}(\pi_{a}(\alpha)) = \mu_{a, p}^{-}(\pi_{a}(\mu_{a, p}^{-}(x)) = \mu_{a, p}^{-}(x) = \alpha$ and $\mu_{a, p}^{+}(\pi_{a}(\beta)) = \mu_{a, p}^{+}(\pi_{a}(\mu_{a, p}^{+}(y)) = \mu_{a, p}^{+}(y) = \beta$.  This, in combination with \eqref{eq:same_itinerary}, implies that $\mu_{a, p'}^{-}(\pi_{a}(\alpha)) = \alpha$ and $\mu_{a, p'}^{+}(\pi_{a}(\beta)) = \beta$,  for all $p' \in (\pi_{a}(\alpha),\pi_{a}(\beta)) \cap (1-a^{-1}, a^{-1})$.  Hence, $\alpha \in \Omega_{a, p'}^{-}$ and $\beta \in \Omega_{a, p'}^{+}$.  It follows from \Cref{thm:Main1.5} that $\alpha \in [\overline{0}, \mu_{a, p'}^{-}(p')] \cup (\mu_{a, p'}^{+}(p'), \overline{1}]$.  (We remind the reader that $\overline{0}$ denotes the element $0 \, 0 \,0 \, \cdots \in \Omega$ and $\overline{1}$ to denotes the element $1 \,1 \,1 \, \cdots \in \Omega$.)  Since $\alpha$ begins with $01$, it must be the case that $\alpha\in [\overline{0}, \mu_{a, p'}^{-}(p')]$.  Moreover, $\alpha \neq \mu_{a, p'}^{-}(p')$, since if $\alpha = \mu_{a, p'}^{-}(p')$, then  by \Cref{prop:code-map} we would have that $\pi_{a}(\alpha) = \pi_{a}( \mu_{a, p'}^{-}(p')) = p'$, which contradicts that $p'\in (\pi_{a}(\alpha),\pi_{a}(\beta)) \cap (1-a^{-1}, a^{-1})$. A similar argument shows that $\beta \in (\mu_{a, p'}^{+}(p'), \overline{1}]$. Therefore, $\alpha \prec \mu_{a, p'}^{-}(p')$ and $\beta \succ \mu_{a, p'}^{+}(p')$, for all $p' \in (\pi_{a}(\alpha),\pi_{a}(\beta)) \cap (1-a^{-1}, a^{-1})$.

(ii) $\Rightarrow$ (iii) This is an immediate consequence of \Cref{thm:Main1.5}.

(iii) $\Rightarrow$ (i) If  $\Omega_{q}^{-} \subset \Omega_{a, p}^{-} $ and $ \Omega_{q}^{+} \subset \Omega_{a, p}^{+}$ for $p\in (\pi_{a}(\alpha),\pi_{a}(\beta)) \cap (1-a^{-1}, a^{-1})$, then by \Cref{thm:Main1.5} we have that $\alpha \preceq \mu_{a, p}^{-}(p) \prec \mu_{a, p}^{+}(p) \prec \beta$ or $\alpha \prec \mu_{a, p}^{-}(p) \prec \mu_{a, p}^{+}(p) \preceq \beta$,  and so by \Cref{cor:cor2} we have that $\exp(h(T)) \leq a$. We will now show that $\exp(h(T)) \not= a$ if  $\Omega_{q}^{\pm} \subset \Omega_{a, p}^{\pm} $. In order to reach a contradiction, suppose that $\exp(h(T)) = a$ and that $\Omega_{q}^{+} \subset \Omega_{a, p}^{+}$ and $\Omega_{q}^{-} \subset \Omega_{a, p}^{-}$, for some $p$ belonging to the interval $(\pi_{a}(\alpha),\pi_{a}(\beta)) \cap (1-a^{-1}, a^{-1})$.  Therefore, fix $p$, such that either 
\begin{align}\label{EQ:(iii)->(i)}
\alpha \preceq \mu_{a, p}^{-}(p) \prec \mu_{a, p}^{+}(p) \prec \beta\quad \text{or} \quad \alpha \prec \mu_{a, p}^{-}(p) \prec \mu_{a, p}^{+}(p) \preceq \beta
\end{align}
holds.  By \cite{G} our given Lorenz system $([0,1], T^{\pm})$ is semi-conjucate to some uniform Lorenz system $([0,1], U^{\pm}_{s, p'})$. Moreover, since the semi-conjugacy preserves topological entropy (\Cref{lemma1}) and since by \Cref{thm:uniform} we have that $h(U_{s, p'}^{\pm}) = \ln(s)$, it follows that $s = \exp(h(T)) = a$. Hence, by \Cref{lemma1}, we have that $\Omega_{s, p'}^{\pm} \subseteq\Omega_{q}^{\pm}$ and therefore,
\begin{align}\label{EQ:(iii)->(i)2}
 \mu_{a, p'}^{-}(p')\preceq \alpha \quad \text{and} \quad \beta\preceq\mu_{a, p'}^{+}(p') .
\end{align}
Combining \eqref{EQ:(iii)->(i)2} with \eqref{EQ:(iii)->(i)} and then applying \Cref{lem2} gives a desired contradiction.
 \end{proof}

Before presenting the proof of \Cref{THM:END} we given the following example which illustrates the importance of taking the intersection of $(\pi_{a}(\alpha),\pi_{a}(\beta))$ with the $(1-a^{-1}, a^{-1})$ in \Cref{Main_THM} (ii) and (iii).
 
\begin{example}\label{EXMP:inequalities}
An instance of when the inequality $\pi_{a}(\beta) > a^{-1}$ can occur is when $T^{\pm}$ is a Lorenz map where the first branch is a linear function with gradient close to $1$ and the second branch is a function of high polynomial or exponential growth.  An explicit example of such a map is the Lorenz map with critical point $1/2$ given by the functions $f_{0}(x) \coloneqq 1.001x$ and $f_{1}(x) \coloneqq \exp(x + \ln(2) - 1) - 1$.  In this case the inequality $\pi_{a}(\alpha) < 1- a^{-1}$ is satisfied for $a = 3/2 > \exp(h(T)) \approx 1.00125$.  (This latter value was calculated using an implemented version of the algorithm presented in \Cref{Sec:Algorithm}, with a tolerance $\epsilon = 0.0001$ and a truncation tern $n = 25,000$.)  By reversing the roles of the first and second branch one obtains a Lorenz map with $\pi_{a}(\beta) > a^{-1}$.
\end{example}

\begin{proof}[Proof of \Cref{THM:END}.]
Since the proofs for (i) and (ii) are essentially the same, we only include a proof of (i).  The result is proved by showing the following set of implications (a) $\Rightarrow$ (b) $\Rightarrow$ (c) $\Rightarrow$ (a).

(a) $\Rightarrow$ (b) Let $a \in (\exp(h(T)), 2)$ and suppose that the inequalities given in \Cref{THM:END}~(i)~(b) do not hold for $p = a^{-1}$.  Then by definition we have that $\tau_{q}^{-}(q) = \mu_{a, a^{-1}}(a^{-1}) = 0 \overline{1}$.  An application of \Cref{cor:cor2} then leads to a contradiction to how the parameter $a$ was originally chosen.  The uniqueness follows directly from \Cref{lem2}.

(b) $\Rightarrow$ (c) This is a direct consequence of \Cref{thm:Main1.5} and the fact that $a > \exp(h(T))$.

(c) $\Rightarrow$ (a) The proof is essentially the same as the proof of (iii) $\Rightarrow$ (i) of \Cref{Main_THM}.
\end{proof}

\section{An algorithm to compute the topological entropy of a Lorenz map}\label{Sec:Algorithm}

The numerical computation of topological entropy of one dimensional dynamical systems has received much attention; see for instance \cite{BK,BKSP,FMT,Mil}.  Based on \Cref{Main_THM,THM:END}, we next provide a new algorithm to compute the topological entropy of a Lorenz system.  The algorithm is stated assuming infinite arithmetic precision.  However, with straightforward modifications, the algorithm can be practically implemented.  Such an implementation was used in obtaining the sample results presented at the end of this section.  After the statement of algorithm a proof of its validity is given.  (We remind the reader that $h(T)$ denotes the common value $h(T^{+}) = h (T^{-})$, for a given Lorenz system $([0, 1], T^{\pm})$.)

\noindent	\parbox{15mm}{\textbf{Input:}} \textit{A Lorenz map $T^{\pm}$ with critical point $q$ and a tolerance $\epsilon \in (0, 1)$.}\vskip2mm

\noindent	\parbox{15mm}{\textbf{Output:}} \textit{An estimate to $h(T)$ within a tolerance of $\epsilon$.}
\begin{enumerate}
	\item[(1)] Compute: $\alpha \coloneqq \tau_{q}^{+}(q)$ and $\beta \coloneqq \tau_{q}^{-}(q)$.
	\item[(2)] Initialise: $a_{1} = 1$ and $a_{2} = 2$.
	\item[(3)] Set $a =\frac{ a_{1} + a_{2}}{2}$.
	\item[(4)] If both $\alpha\neq 0\overline{1}$ and $\beta\neq 1\overline{0}$ , then go to Step~(5), else go to Step~(4)(a).
	\begin{enumerate}
	\item[(a)] If both $\alpha= 0\overline{1}$ and $\beta \neq 1\overline{0}$, then compute $\mu^{+}_{a, a^{-1}}(a^{-1})$ and go to Step(11), else go to Step(4)(b).
	\item[(b)] Compute $\mu^{-}_{a,1-a^{-1}}(1-a^{-1})$ and go to Step~(12).
	\end{enumerate}
	\item[(5)] Compute: $\pi_{a}(\alpha)$ and $\pi_{a}(\beta)$.
	\item[(6)] Compute: $t_{1}(a) \coloneqq \max\{\pi_{a}(\alpha),1-a^{-1}\}$ and $t_{2}(a) \coloneqq \min\{\pi_{a}(\beta), a^{-1}\}$.
	\item[(7)] If $t_{1}(a) \geq t_{2}(a)$, then $a_{1} \gets a$ and go to Step(13), else go to Step(8).
	\item[(8)] Set $p = (t_{1}(a) + t_{2}(a))/2$.
	\item[(9)] Compute: $\mu^{+}_{a,p}(p)$ and $\mu_{a, p}^{-}(p)$.
	\item[(10)] If $\alpha \prec \mu^{-}_{a, p}(p)$ and  $\mu^{+}_{a, p}(p) \prec \beta$, then go to Step~(10)(a), else go to Step~(10)(b).
	\begin{enumerate}
	 \item[(a)] $a_{2} \gets a$ and go to Step(13).
	 \item[(b)] $a_{1} \gets a$ and go to Step(13).
	 \end{enumerate}
	 \item[(11)] If $\mu^{+}_{a, a^{-1}}(a^{-1})\prec\beta$, then $a_{2} \gets a$ and go to Step~(13), else $a_{1} \gets a$ and go to Step~(13).
	 \item[(12)] If $\alpha \prec \mu^{-}_{a, 1-a^{-1}}(1-a^{-1})$, then $a_{2} \gets a$ and go to Step(13), else $a_{1} \gets a$ and go to Step~(13).
	 \item[(13)] If $a_{2} - a_{1} < \epsilon/2$, then return $h(T) \in [\ln ((a_{1} + a_{2})/2 - \epsilon/4), \ln ((a_{1} + a_{2})/2 + \epsilon/4)]$ and terminate the algorithm, else go to Step(3).
\end{enumerate}

\begin{proof}[Proof of the validity of the Algorithm.]
The variable $a$ in the algorithm is the midpoint of the interval $[a_{1},a_{2}]$ which is initialized at $[a_{1},a_{2}] = [1, 2]$, and thus, $\ln (a_{1}) \leq h(T) < \ln (a_{2})$.   We will show that, throughout the algorithm, the following inequality is maintained, 
\begin{align}\label{eq:alg}
\ln (a_{1}) \leq h(T) \leq \ln (a_{2}).
\end{align}
A tolerance $\epsilon >0$ is fixed at the start.  At each iteration (Step~(3) to Step~(13)) of the algorithm, the length of this interval $[a_{1}, a_{2}]$ is halved until, at Step~(13), we arrive at $a_{2} - a_{1} < \epsilon/2$.  According to \eqref{eq:alg}, at this point we have estimated
the entropy within the desired tolerance $\epsilon \in (0, 1)$, specifically
\begin{align*}
 \ln \left( (a_{1} + a_{2})/2 - \epsilon/4 \right) \leq h(T) \leq \ln \left( (a_{1} + a_{2})/2 + \epsilon/4 \right).
\end{align*}
Suppose, in Step~(4), that $\alpha\neq 0\overline{1}$ and $\beta\neq 1\overline{0}$, namely, that the critical point $q$ is such that $f_{0}(q) \neq 1$ and $f_{1}(q) \neq 0$.  (Here, we remind the reader that $f_{0}: [0, q] \to [0, 1]$ and $f_{1}: [q,1] \to [0, 1]$ are the expanding maps which define the given $T^{\pm}$; see \Cref{Defn:Lorenz_map}.) At Step~(7) or Step~(10) the interval $[a_{1}, a_{2}]$ will be replaced by either $[a_1,a]$ or $[a, a_{2}]$, where $a$ has the value $(a_{1} + a_{2})/2$. It will now be proved that at each iteration (Step~(3) to Step~(13)), the inequalities given in \eqref{eq:alg} are maintained.  To see this we will follow the steps of the algorithm. At Step~(3), the value of $a$ is set to the value of the midpoint of the interval $[a_{1}, a_{2}]$. In Step~(5), the images of the critical itineraries  $\alpha$ and $\beta$ of the given Lorenz system $([0,1], T^{\pm})$ under $\pi_{a}$ are computed. In Step~(6), the values of $t_{1}(a)$ and $t_{2}(a)$ are set to the left and right endpoints, respectively, of an interval which, according to \Cref{lem:bounds}, has non-empty interior provided that $h(T) < \ln (a)$.  Thus, if $t_{1}(a) \geq t_{2}(a)$, then $h(T) \geq \ln (a)$.  In this case, the value of $a_{1}$ is reset to the value of $a$ in Step~(7) and the inequalities in \eqref{eq:alg} are maintained. The algorithm then proceeds to Step~(13).

On the other hand, if $t_{2}(a) > t_{1}(a)$, then in Step~(8) the value of $p$ is set to the midpoint of the interval $[t_{1}(a), t_{2}(a)]$.  In Step~(9) the algorithm computes  the critical itineraries,  $\mu^{+}_{a, p}(p)$ and $\mu_{a, p}^{-}(p)$,  of the uniform Lorenz systems $([0, 1], U_{a, p}^{\pm})$. In Step~(10) the algorithm compares $\mu^{-}_{a, p}(p)$ with $\alpha$ and compares $\mu^{+}_{a, p}(p)$ with $\beta$.  There are two possibilities, either both $\mu_{a, p}^{-}(p) \succ \alpha$ and $\mu_{a, p}^{+}(p) \prec \beta$ hold or not.  
\begin{enumerate}
\item If $\mu_{a, p}^{-}(p) \succ \alpha$ and  $\mu_{a, p}^{+}(p) \prec \beta$, then $h(T) \leq \ln (a)$, see \Cref{cor:cor2}. Therefore, to maintain the inequalities given in \eqref{eq:alg}, the value of $a_{2}$ is reset to the value of $a$.
\item Otherwise, we have $h(T) \geq \ln (a)$.  Since, if this was not the case, then this would contradict \Cref{Main_THM}.  Therefore, to maintain the inequalities given in \eqref{eq:alg}, the value of $a_{1}$ is reset to the value of $a$.
\end{enumerate}
In either of the above two case, the algorithm then proceeds to Step~(13).

Returning to Step~(4), suppose that $\alpha = 0\overline{1}$ and $\beta \neq 1\overline{0}$.  Observe, for each $a \in (1, 2)$, that $\mu_{a, a^{-1}}^{-}(a^{-1}) = \alpha = 0\overline{1}$.  There are now two possibilities, either $\mu_{a,a^{-1}}^{+}(a^{-1}) \prec \beta$ or not.
\begin{enumerate}\setcounter{enumi}{2}
\item If $\mu^{+}_{a, a^{-1}}(a^{-1}) \prec \beta$, then, by \Cref{cor:cor2} and since $\mu_{a, p}^{-}(p) = \alpha = 0\overline{1}$, we have that $h(T) \leq \ln(a)$.  Therefore, to maintain the inequalities given in \eqref{eq:alg}, the value of $a_{2}$ is reset to the value of $a$.  The algorithm then proceeds to Step~(13).  
\item If $\mu^{+}_{a, a^{-1}}(a^{-1}) \succeq \beta$, then, by \Cref{cor:cor2} and since $\mu_{a, p}^{-}(p) = \alpha = 0\overline{1}$, we have that $h(T) \geq \ln(a)$.  Therefore, to maintain the inequalities given in \eqref{eq:alg}, the value of $a_{1}$ is reset to the value of $a$.  The algorithm then proceeds to Step~(13).
\end{enumerate}
At Step~(13), provided that $a_{2} - a_{1} \geq \epsilon/2$, the algorithm proceeds to the next iteration, otherwise the algorithm returns the following value and terminates: $h(T^{+}) = h(T^{-}) \in [\ln ((a_{1} + a_{2})/2 - \epsilon/4), \ln ((a_{1} + a_{2})/2 + \epsilon/4)]$.  Similarly, if $\alpha \neq 0\overline{1}$ and $\beta = 1\overline{0}$, then in Step~(4)(b) of the algorithm the value of $p$ is set to $1-a^{-1}$ and the itinerary $\mu_{a, 1-a^{-1}}^{-}(1-a^{-1})$ is computed.  The algorithm then proceeds to Step~(12), where, to maintain the inequalities in \eqref{eq:alg}, the algorithm either 
\begin{enumerate}\setcounter{enumi}{4}
\item resets the value of $a_{2}$ to the value of $a$, if $\alpha \prec \mu^{-}_{a, 1-a^{-1}}(1-a^{-1})$, or
\item resets the value of $a_{1}$ to the value of $a$, if $\alpha \succeq \mu^{-}_{a, 1-a^{-1}}(1-a^{-1})$.
\end{enumerate}
The algorithm then goes to Step~(13); here it either goes to the next iteration or terminates.

Observe that the situation where $\alpha = 0\overline{1}$ and $\beta = 1\overline{0}$ cannot occur, since by definition of the itineraries, this would immediately imply that $f_{0}(q) = 1$ and $f_{1}(q) = 0$.  Thus, the given system is not a Lorenz system as it would violate condition (i) of \Cref{Defn:Lorenz_map}.
\end{proof}

\subsection{Sample results.}\label{Sec:Sample_results}

Presented below are two examples that demonstrates an implemented version of our algorithm.  These examples indicate that the algorithm returns an accurate estimate for the entropy of a Lorenz system.  To practically implement the algorithm, itineraries are computed to a prescribed length $n \geq 3$, which is called the \textit{truncation term} and is an additional input to the algorithm.  

\begin{example}\label{exmp:exampl1}
Consider the Lorenz map $T^{\pm}$ with critical point $q$ given by $f_{0}(x) = a\sqrt{x}$ and $f_{1}(x) = b x  + 1 - b$, where $a = 1.25$, $b = (a^{-6}-1)/(a^{-2}-1)$ and $q =1/a^{2}$.  The reason for this choice of $a, b, q$ is that, in this case, there is a theoretical method for determining the topological entropy of the map $T^{\pm}$. This allows us to compare the estimated value for the entropy given by our algorithm to the actual value.  To be more precise, to theoretically determine  the topological entropy we use the fact that, for this choice of $a,b,q$, the critical itineraries are periodic and therefore this Lorenz map is Markov. For Markov maps the topological entropy is the logarithm of the maximum eigenvalue of the associated adjacency matrix \cite[Proposition~3.4.1]{BS}.  Using this method we obtain that $h(T^{\pm}) = \ln((1+\sqrt{5})/2) \approx 0.4812118251$.  The following table gives the output of a practically implemented version of our algorithm for this map, where $\epsilon$ denote the tolerance term and $n$ denotes the truncation term.
\begin{small}
\begin{table}[h]
\begin{tabular}{lccc}
\hline
& $\epsilon=10^{-2}$ & $\epsilon=10^{-4}$ & $\epsilon=10^{-6}$\\
\hline
$n = 10$ & 0.4831010758 & 0.4811979105 & 0.4812117615\\
$n = 100$ & 0.4831010758 & 0.4811979105 & 0.4812117615\\
$n = 1,000$ & 0.4831010758 & 0.4811979105 & 0.4812117615\\
$n = 10,000$ & 0.4831010758 & 0.4811979105 & 0.4812117615\\
\hline
\end{tabular}
\end{table}
\end{small}
\end{example} 

\begin{example}
Here we consider the uniform Lorenz map $U^{\pm}_{a, 1/2}$ and the uniform Lorenz map $U^{\pm}_{a, a^{-1}}$ for $a = \sqrt{2}$, which, by \Cref{lemma1}, both have topological entropy equal to $\log(\sqrt{2}) \approx 0.34657359023$.  The following table gives the output of a practically implemented version of our algorithm for these maps, where $\epsilon$ denote the tolerance and $n$ denotes the truncation term.
\begin{small}
\begin{table}[h]
\begin{tabular}{lcccc}
\hline
& \multicolumn{2}{c}{$p=1/2$} & \multicolumn{2}{c}{$p=a^{-1}=1/\sqrt{2}$}\vspace{2mm}\\
& $\epsilon=10^{-3}$ & $\epsilon=10^{-6}$ & $\epsilon=10^{-3}$ & $\epsilon=10^{-6}$\\
\hline
$n = 10$ & 0.3652803888 & 0.3655560121 & 0.3475021428 & 0.3471925188\\
$n = 100$ & 0.3468120116 & 0.3465736575 & 0.3468120116 & 0.3465736575\\
$n = 1,000$ & 0.3468120116 & 0.3465736575 & 0.3468120116 & 0.3465736575\\
$n = 10,000$ &  0.3468120116 & 0.3465736575 & 0.3468120116 & 0.3465736575\\
\hline
\end{tabular}
\end{table}
\end{small}
\end{example} 

\section*{Acknowledgements}

The first author was supported by {\footnotesize EPSRC:EP/PHDPLUS/AMC3/ DTG2010} and partially by {\footnotesize ARC:DP0984353}.  The second author was partially supported by {\footnotesize ARC: DP0558974}.  Our thanks also go to M.\ Barnsley for hosting us at ANU and for providing the initial motivation, and to J.\ Keesling for his initial indispensable suggestions.  Finally the authors thank the unknown referees for their valuable suggestions.

\bibliographystyle{plain}
\bibliography{bib}

\end{document}